\documentclass[11pt]{article}
\usepackage{mathbbold}
\usepackage{latexsym,amssymb,amsmath,amsthm}
\usepackage{mathrsfs}
\usepackage{amsmath}
\usepackage{amsthm}
\usepackage{amsfonts}
\usepackage{amssymb}
\usepackage{latexsym}
\usepackage[all]{xy}
\def \Hom{\operatorname{Hom}}

\def \Ext{\operatorname{Ext}}\def \End{\operatorname{End}}

\def \mod{\operatorname{mod}}\def \Mod{\operatorname{Mod}}
\def \add{\operatorname{add}}\def \Add{\operatorname{Add}}
\def \End{\operatorname{End}}
\def \add{\operatorname{add}}

\def \Prod{\operatorname{Prod}}

\def\C{\mathscr{C}}
\def\D{\mathscr{D}}
\def\G{\mathcal{G}}
\def\T{\mathrm{T}}
\def\Im{\mathop{\rm Im}\nolimits}

\def\mod{\mathop{\rm mod}\nolimits}
\def\Mod{\mathop{\rm Mod}\nolimits}
\def\Hom{\mathop{\rm Hom}\nolimits}

\def\Ext{\mathop{\rm Ext}\nolimits}

\def\add{\mathop{\rm add}\nolimits}

\newtheorem{theorem}{Theorem}[section]
\newtheorem{lemma}[theorem]{Lemma}
\newtheorem{proposition}[theorem]{Proposition}
\newtheorem{corollary}[theorem]{Corollary}
\newtheorem{definition}[theorem]{Definition}

\newtheorem{remark}[theorem]{Remark}

\begin{document}

\begin{center}
{\large  \bf Gorenstein categories and separable equivalences}
%\footnote {}
\\ \vspace{0.6cm} {Guoqiang Zhao}
\\ \footnotesize {\it Department of Mathematics, Hangzhou Dianzi University\\
Hangzhou, 310018, China\\ gqzhao@hdu.edu.cn}\\ 
 \vspace{0.6cm} {Juxiang Sun}
 \footnote {Corresponding author.}\\
\footnotesize {\it School of Mathematics and Statistics, Shangqiu Normal University\\
Shangqiu 476000, P. R. China\\ Sunjx8078@163.com}\\
\end{center}

\baselineskip=18pt
%\maketitle{}

\noindent {\bf Abstract.}
Let $\C$ be an additive subcategory of left $\Lambda$-modules, 
we establish relations of the orthogonal classes of $\C$ and (co)res $\widetilde{\C}$ under separable equivalences.
As applications, we obtain that the (one-sided) Gorenstein category
and Wakamatsu tilting module are preserved under separable equivalences.
Furthermore, we discuss when $G_{C}$-projective (injective) modules and 
Auslander (Bass) class with respect to $C$ are invariant under separable equivalences.
\vspace{0.2cm}



\section{Introduction}
  
As a common generalization of Morita equivalences and the stable equivalences of Morita type, 
the notion of the separable equivalence was originally introduced by Linkelmans in \cite{L}. 
There are examples of algebras that are separably equivalent but not equivalent mentioned above.
%and has many interesting applications (see \cite{BE, BT, L, P1} for detail).
It has been known that two rings, which are (symmetrically) separably equivalent,
share many important properties,
such as global dimensions, finitistic dimensions, complexity, representation dimensions,
representation type  and so on (see \cite{BE, K, L, P, X1} for detail).

Gorenstein projective, Gorenstein injective modules and related modules are central in relative homological algebra (\cite{C, EJ1, EJ2, Ho, Wh}). 
As a deeper development, these notions were extended to the Gorenstein category $G(\C )$ by Sather-Wagstaff et al. \cite{SSW} 
and right (left) Gorenstein category $r\mathcal{G}(\C)$ ($l\mathcal{G}(\C)$) by Song et al. \cite{SZH}, 
where $\C$ is a subcategory of an abelian category.
For the case of $\C=\add T$, where $T$ is a finitely generated and self-orthogonal module, 
right Gorenstein category has been studied by Auslander and Reiten in \cite{AR} (with different notation $\mathcal{X}_{\C}$), 
and was showed that this category is resolving and functorially finite provided that $T$ is a cotilting module. 
Afterwards, these results were generalized to the case that $T$ is a Wakamatsu tilting module by Mantese and Reiten \cite{MR}.

In \cite{LX}, Liu and Xi started to study the relations between Gorenstein projective modules under the stable equivalence of Morita type,
and proved that Gorenstein projective modules are preserved under the stable equivalence of Morita type with two adjoint pairs.
The proof depended on the condition that having two adjoint pairs.
%Recently, more invariants of stable equivalences are obtained by Xi and Zhang \cite{XZ}.
The purpose of this paper is to remove this restricted condition,
and study whether Gorenstein categories are preserved under any separable equivalence.
Some applications are also given.

The paper is organized as follows. 
In Section 2, we give some terminology and preliminary results which are often used in this paper.
In Section 3, for an additive subcategory $\C$ of left $\Lambda$-modules, 
we first prove that the orthogonal classes of $\C$ and (co)res $\widetilde{\C}$ are preserved under any separable equivalence.
As applications, we obtain that the (one-sided) Gorenstein category and Wakamatsu tilting module are preserved under separable equivalences. 
Furthermore, we also investigate the invariant properties of $G_{C}$-projective (injective) modules and Auslander (Bass) class 
with respect to $C$ under separable equivalences, 
and obtain more invariants for separable equivalences.

\vspace{0.5cm}

\section{Preliminaries}

Throughout this paper,  all rings are nontrivial  associative rings, and all modules are  left modules unless stated otherwise.
Let $\Lambda$ be a ring, we denote by $\Mod \Lambda$ (resp. $\mod \Lambda$) 
the category of all  (resp. finitely generated) left $\Lambda$-modules.
Denoted by $\mathcal{P}(\Lambda)$ (resp. $\mathcal{I}(\Lambda)$) the subcategory of 
$\Mod \Lambda$ consisting of all projective (resp. injective)$\Lambda$-modules.
Let $M$ be a $\Lambda$-module,  we denote by $\Add M$ (resp. $\add M$) 
the subcategory of $\Mod \Lambda$ (resp. $\mod \Lambda$) 
consisting of all direct summands of (resp. finite) direct sums of copies of $M$, 
and $\Prod M$ the subcategory of $\Mod \Lambda$ 
consisting of all direct summands of direct products of copies of $M$.

%Peacock in [21] introduced the notion of separable equivalences between exact categories as follows.

%\begin{definition} {\rm (\cite{P}}
%Let $\A$ and $\B$ be abelian categories. $\A$ and $\B$ are said to be separably equivalent, 
%if there exist exact functors $F : \A \to \B$ and $G : \B \to \A$ such that

%$(1)$ the image of a projective object is projective;

%$(2)$ the identity functor is a summand of $GF$ and of $FG$.
%\end{definition}

%Two rings $\Lambda$ and $\Gamma$ are separably equivalent if and only if their module categories are separably equivalent.

\begin{definition} {\rm (\cite{BE, L}}
Let $\Lambda$ and $\Gamma$ be two rings, 
and let $_{\Lambda}M_{\Gamma}$ and $_{\Gamma}N_{\Lambda}$
be finitely generated and projective as one-sided modules.
We say that $\Gamma$ separably divides $\Lambda$, if there are $\Lambda$-bimodule $_{\Lambda}X_{\Lambda}$ and
$\Lambda$-bimodule isomorphism $M\otimes_{\Gamma}N\cong \Lambda\oplus X$.

If, in addition,  there exist $\Gamma$-bimodule $_{\Gamma}Y_{\Gamma}$ and  $\Gamma$-bimodule isomorphism    $N\otimes_{\Lambda}M\cong \Gamma\oplus Y$,
then $\Lambda$ and $\Gamma$ are said to be separably equivalent.

%$(2)$ We say that $\Gamma$ symmetrically separably divides $\Lambda$, if $\Gamma$ separably divides $\Lambda$ and $(M\otimes_{\Gamma}-, N\otimes_{\Lambda}-)$ is an adjoint pair between $\Lambda$-modules and $\Gamma$-modules. Moreover, if $\Lambda$ symmetrically separably divides $\Gamma$ linked by  $_{\Lambda}M_{\Gamma}$ and $_{\Gamma}N_{\Lambda}$ such that $(N\otimes_{\Lambda}-, M\otimes_{\Gamma}-)$ also is an adjoint pair, then $\Lambda$ and $\Gamma$ are said to be symmetrically separably equivalent. 
\end{definition}

\begin{remark} 
{\rm Let $\Lambda$ and $\Gamma$ be separably equivalent defined as in Definition 2.1. 

 $(i)$ If $_\Lambda X_\Lambda$ and $_\Gamma Y_\Gamma$ are the zero modules, 
 then $\Lambda$ is Morita equivalent to $\Gamma$;

$(ii)$ If  $_\Lambda X_\Lambda$ and $_\Gamma Y_\Gamma$ are projective as bimodules, 
then $\Lambda$ and $\Gamma$ are stably equivalent of Morita type.}
\end{remark}

%\begin{example} \cite{BT,L,P1} 
%(1) For a ring $\Lambda$, $M_n(\Lambda)$ ( the matrix ring of $\Lambda$ degree in $n$ ) and $\Lambda$ are separably equivalent;

%(2) Let $k$ be a field of characteristic $p>0$, and let $G$ be a finite group. If $H$ is a Sylow $p$-subgroup of $G$, then $kG$ and $kH$ are separably equivalent;

% (3) Let $\Lambda$ be a ring and $G$ be a finite group. If $|G|^{-1}\in \Lambda$, then the skew group ring $\Lambda G$ is symmetrically separably equivalent to $\Lambda$;

% (4) Let $\Lambda$ be a finite dimensional $k$-algebra, and  $F$ be  a finite separable field extension of $k$, then $\Lambda$ and $\Lambda\otimes_{k}F$ are symmetrically separably equivalent.
%\end{example}

%More examples of (symmetric) separable equivalences can be found in \cite{P1}.

Suppose that $\Lambda$ and $\Gamma$ are separably equivalent defined as in Definition 2.1.
In the following, we write $\mathrm{T}_M=M\otimes_{\Gamma}-: \Mod\Gamma\to \Mod\Lambda$,
and $\mathrm{T}_N=N\otimes_{\Lambda}-: \Mod\Lambda\to \Mod\Gamma$.
Similarly, we have the functors $\mathrm{T}_X$ and $\mathrm{T}_Y$.

\begin{lemma} 
Suppose that $\Lambda$ and $\Gamma$ are separably equivalent defined by $_{\Lambda}M_\Gamma$ and $_\Gamma N_\Lambda$,
then

$(1)$ $_\Lambda M_\Gamma$ and $_\Gamma N_{\Lambda}$ are projective generators as one-sided modules;

$(2)$ $\mathrm{T}_M$ and $\mathrm{T}_N$ are exact functors such that the image of a projective module is projective;
%$\mathrm{T}_X$, $\mathrm{T}_Y$: projective module to projective.

$(3)$ $\mathrm{T}_M \circ  \mathrm{T}_N\to \mathrm{Id}_{\Mod \Lambda }\oplus \mathrm{T}_X$ and
$\mathrm{T}_N \circ  \mathrm{T}_M\to \mathrm{Id}_{\Mod \Gamma}\oplus \mathrm{T}_Y$ are natural isomorphisms. 
\end{lemma} 

\begin{proof}
(1) We only prove that $_{\Lambda}M$ is a projective generator of $\Mod \Lambda$.
Because $_\Gamma N$ is a finitely generated projective $\Gamma$-module, we have $_{\Gamma}N\in \add_{\Gamma}\Gamma$.
Thus $\Lambda\in\add_{\Lambda}(M\otimes_\Gamma N)\subseteq \add_{\Lambda}(M\otimes_{\Gamma}\Gamma)\subseteq \add _{\Lambda}M$ as ragarded.

(2) and (3) are clear.
\end{proof}

Let $\C$ be an additive subcategory of an abelian category $\mathscr{A}$. 
Recall that an exact sequence in $\mathscr{A}$ is said to be 
{\it $\Hom_{\mathscr{A}}(-, \C)$-exact} (resp. {\it $\Hom_{\mathscr{A}}(\C, -)$-exact}), 
if it remains exact after applying the functor
$\Hom_{\mathscr{A}}(-, C)$ (resp. {\it $\Hom_{\mathscr{A}}(C, -)$-exact}) for all $C\in \C$. 
Following \cite{SSW}, we write 
\begin{center}
$\C{^\bot}= $  $\{A\in \mathscr{A}$ $|$ $\Ext^{\geq 1}_{\mathscr{A}}(C, A)= 0$ for all $C\in\C$ $\}$,
\end{center}
 and
 \begin{center}
$\mathrm{res} \widetilde{\C}$ = $\{A\in \mathscr{A}$ $|$ there exists a $\Hom_{\mathscr{A}}(\C , -)$-exact exact sequence 
$\cdots \to C_{i}\to\cdots\to C_{1}\to C_{0}\to A\to 0$ in $\mathscr{A}$ with all $C_{i}\in \C \}$. 
\end{center}

${^\bot}\C$ and $\mathrm{cores} \widetilde{\C}$ can be defined dually.
$\C$ is called self-orthogonal if $\C\subseteq {^\bot}\C$.

\begin{definition} {\rm(\cite{SSW})
The Gorenstein category $\mathcal{G}(\C)$ is defined as
$\mathcal{G}(\C) =\{A\in$ $\mathscr{A}$ $\mid$
there exists an exact sequence: $$\cdots \to C_1 \to C_0 \to C^0 \to
C^1 \to \cdots $$ in $\C$, which is both
$\Hom_{\mathscr{A}}(\C,-)$-exact and
$\Hom_{\mathscr{A}}(-,\C)$-exact, such that $A\cong
\Im(C_0\to C^0)\}$.}
\end{definition}

\begin{definition} {\rm(\cite{SZH})
We call $r\mathcal{G}(\C) = {^\bot}\C \cap \mathrm{cores} \widetilde{\C}$ and
$l\mathcal{G}(\C) = \C{^\bot} \cap \mathrm{res} \widetilde{\C}$  
the right and left Gorenstein subcategory of $\mathscr{A}$ relative to $\C$, respectively.}
\end{definition}

\begin{remark}
{\rm Let $\mathscr{A}= \Mod \Lambda$, from \cite[Example 3.2]{SZH} we have

 (1) If $\C=\mathcal{P}(\Lambda)$ (resp. $\mathcal{I}(\Lambda)$), 
then $r\G(\C)$ (resp. $l\G(\C)$) is the subcategory of Mod $\Lambda$ consisting of Gorenstein projective (resp. injective) modules.

(2) Let $_{\Lambda}C$ be a Wakamatsu tilting module with $\Gamma= \End{_\Lambda}C$.

(i) If $\C=\Add C$, then $r\G(\C)$ is the subcategory $\mathcal{GP}_{C}(\Lambda)$ of Mod $\Lambda$ consisting of $G_{C}$-projective modules, 
and $l\G(\C)$ is the Bass class $\mathscr{B}_{C}(\Lambda)$ with respect to $C$.
 
 (ii) If $\C =\Prod C^{+}$, where $C^{+} = \Hom_{Z}(C, Q/Z)$,
then $l\G(\C)$ is the subcategory $\mathcal{GI}_{C}(\Gamma)$ of Mod $\Gamma$ consisting of $G_{C}$-injective modules, 
and $r\G(\C)$ is the Auslander class $\mathscr{A}_{C}(\Gamma)$ with respect to $C$.}
\end{remark}

%Recall from \cite{BM} that {\it the left Gorenstein global dimension} of a ring $\Lambda$,
%denoted by l.Ggl.dim$\Lambda$, is defined as sup$\{\Gpd_{\Lambda}K |K\in \Mod \Lambda \}.$

\begin{definition} {\rm(\cite{MR})
A self-orthogonal module $_{\Lambda}C$ is called Wakamatsu (or generalized) tilting, if the following conditions are satisfied:

(1) $_{\Lambda}C\in$ $\mathrm{res} (\widetilde{\add \Lambda})$.

(2) $_{\Lambda}\Lambda\in$ $\mathrm{cores} (\widetilde{\add C})$.}
%there exists an exact sequence
%$ 0\to {_\Lambda}\Lambda\to T_0 \to T_1\to \cdots \to T_m \to  0$
%in $\mod \Lambda$ with all $T_i \in \add T.$ 
\end{definition}  

By \cite[Corollary 3.2]{Wa}, we have that $_{\Lambda}C$ is Wakamatsu tilting with $\Gamma = \End(_{\Lambda}C)$ if and only if 
$C_{\Gamma}$ is Wakamatsu tilting with $\Lambda = \End(C_{\Gamma})$. 
In particular, it is trivial that $_{\Lambda}\Lambda$ is a Wakamatsu tilting module.

\section{Main Results}

From now on, we assume that $\Lambda$, $\Gamma$, $M$, $N$, $X$ and $Y$ are fixed, as in Definition 2.1, 
and $\C$ is an additive subcategory of $\Mod \Lambda$ closed under isomorphisms.
%In the following, we study the properties of $r\G(\C)$ under separable equivalences. 
%The dual versions of all the results on $r\G(\C )$ also hold true on $l\G(\C)$ by using completely dual arguments.

We begin with the following two lemmas, which play an important role in the sequel.

\begin{lemma}
Let $\Lambda$ and $\Gamma$ be separably equivalent defined by $_{\Lambda}M_\Gamma$ and $_\Gamma N_\Lambda$,
and let $G$ be a $\Lambda$-module. Then 

$(1)$ $G\in$ ${^\bot}\C$ if and only if $\T_N(G) \in$ ${^\bot}\T_N(\C)$. 

$(2)$ $G\in$ $\C{^\bot}$ if and only if $\T_N(G) \in$ $\T_N(\C){^\bot}$. 

In particular, $\C$ is self-orthogonal if and only if so is $\T_N(\C)$.
\end{lemma}

\begin{proof}
(1) Suppose $G\in$ ${^\bot}\C$, and
let $$\cdots\to P_{1}\stackrel{f_{1}}\to P_0\stackrel{f_{0}}\to G\to 0 \qquad(3.1)$$
be a projective resolution of $G$.
From Lemma 2.3(2), applying $\T_N$ to (3.1) gives rise to a projective resolution of $\T_N(G)$
$$\cdots\to \T_N(P_{1})\stackrel{\T_N(f_{1})}\to \T_N(P_0)\stackrel{\T_N(f_{0})}\to \T_N(G)\to 0 \qquad(3.2)$$
We claim that the sequence (3.2) is $\Hom_{\Gamma}(-, \T_N(C))$-exact for any $C\in\C$,
which means that $\T_N(G) \in$ ${^\bot}\T_N(\C)$.
In fact, for each $i\geq 0$, let $K_{i} = \Im f_{i}$, from the exact sequence
$0\to K_{i+1}\stackrel{\iota_{i}}\rightarrow P_{i}\stackrel{\pi_{i}} \to K_{i}\to 0$
one has the exact sequence
$$0\to \T_N(K_{i+1})\stackrel{\T_N(\iota_{i})}\rightarrow \T_N(P_{i})\stackrel{\T_N(\pi_{i})} \to \T_N(K_{i})\to 0 \qquad(3.3)$$

Suppose $g\in\Hom_{\Gamma}(\T_N(K_{i+1}), \T_N(C))$. 
Then $\T_M(g)\in$ $\Hom_{\Lambda}(\T_{M}\T_N(K_{i+1}), \T_{M}\T_N(C))$ $\cong$
$\Hom_{\Lambda}(K_{i+1}\oplus \T_{X}(K_{i+1}), C\oplus \T_{X}(C))$.
Let $\T_M(g) = \left(\begin{array}{cc}
                \alpha_{1} & \alpha_{2} \\
                 \alpha_{3} & \alpha_{4} \\
               \end{array}
             \right)$, then $\alpha_{1}\in$ $\Hom_{\Lambda}(K_{i+1}, C)$.
Because $G\in$ ${^\bot}\C$, the sequence (3.1) is $\Hom_{\Lambda}(-, C)$-exact,
and hence $\Hom_{\Lambda}(\iota_{i}, C):$ $\Hom_{\Lambda}(P_{i}$, $C)\to \Hom_{\Lambda}(K_{i+1}, C)$ is an epimorphism.
Thus there exists $h\in$ $\Hom_{\Lambda}(P_{i}, C)$, such that $\alpha_{1} = h\iota_{i}$.

On the other hand, we have
 $$\T_N\T_M(g)=
\left(\begin{array}{cc}
                g & 0 \\
                 0 & \mathrm{T}_Y(g) \\
               \end{array}
             \right)$$
            
Therefore $g = \T_N(\alpha_{1}) = \T_N(h)\T_N(\iota_{i})$,  
which implies that the sequence (3.3) is $\Hom_{\Gamma}(-$, $\T_N(C))$-exact,
and so we obtain the claim.               
               
Conversely, assume $\T_N(G) \in$ ${^\bot}\T_N(\C)$.
By a similar argument of the \lq\lq if\rq\rq part, for any $C\in \C$, one has $\T_M\T_N(G) \in$ ${^\bot}\T_M\T_N(C)$.
Since $\T_M\T_N = \mathrm{Id}_{\Mod \Lambda }\oplus \T_X$, then $G\in$ ${^\bot}(C\oplus\T_X(C))$,
and so $\Ext^{\geq 1}_{\Lambda}(G, C)= 0$, as desired.    

(2) From the proof of \cite[Theorem 4.1]{X}, 
the conclusion that the image of a injective module under the functor $\T_N$ is injective
remains true in our setting. By a dual argument of (1), one has the assertion.     
\end{proof}

\begin{lemma}
Let $\Lambda$ and $\Gamma$ be separably equivalent linked by $_{\Lambda}M_\Gamma$ and $_\Gamma N_\Lambda$.

$(1)$ If $A\in$ $\mathrm{cores} \widetilde{\C}$, then $\T_N(A) \in$ $\mathrm{cores} \widetilde{\T_{N}(\C)}$.

$(2)$ If $A\in$ $\mathrm{res} \widetilde{\C}$, then $\T_N(A) \in$ $\mathrm{res} \widetilde{\T_{N}(\C)}$.
\end{lemma}

\begin{proof}
We only prove (1), and (2) can be proved dually.

Suppose $A\in$ $\mathrm{cores} \widetilde{\C}$, then there is a $\Hom_{\Lambda}(-, \C)$-exact sequence
$$0\to A\stackrel{f^{0}}\to C^{0}\stackrel{f^{1}}\to C^{1}\to \cdots \qquad(3.4)$$

From Lemma 2.3(2), applying $\T_N$ to (3.4) gives rise to an exact sequence
$$0\to \T_N(A)\stackrel{\T_N(f^{0})}\to \T_N(C^{0})\stackrel{\T_N(f^{1})}\to \T_N(C^{1})\to \cdots \qquad(3.5)$$
It is suffice to show that the sequence (3.5) is $\Hom_{\Gamma}(-, \T_N(C))$-exact for any $C\in\C$.
Indeed, for each $i\geq 1$, let $K^{i} = \Im f^{i}$, then $ \Im \T_{N}(f^{i})\cong$ $\T_{N}(K^{i})$.
Consider the exact sequence
$$0\to \T_N(A)\stackrel{\T_N(f^{0})}\rightarrow \T_N(C{^{0}})\stackrel{\T_N(\pi)} \to \T_N(K^{1})\to 0 \qquad(3.6)$$

Let $g\in\Hom_{\Gamma}(\T_N(A), \T_N(C))$. 
Then $\T_M(g)\in$ $\Hom_{\Lambda}(\T_{M}\T_N(A), \T_{M}\T_N(C))$ $\cong$
$\Hom_{\Lambda}(A$ $\oplus \T_{X}(A), C\oplus \T_{X}(C)$.
Write $\T_M(g) = \left(\begin{array}{cc}
                \beta_{1} & \beta_{2} \\
                 \beta_{3} & \beta_{4} \\
               \end{array}
             \right)$, then $\beta_{1}\in$ $\Hom_{\Lambda}(A, C)$.
Because the sequence (3.4) is $\Hom_{\Lambda}(-, C)$-exact,
the sequence $\Hom_{\Lambda}(C^{0}, C)\stackrel{\Hom_{\Lambda}(f^{0}, C)}\longrightarrow \Hom_{\Lambda}(A, C)\to 0$ is exact.
Thus there exists $h\in$ $\Hom_{\Lambda}(C^{0}, C)$ such that $ \beta_{1} = hf^{0}$.

On the other hand, we have
 $$\T_N\T_M(g)=
\left(\begin{array}{cc}
                g & 0 \\
                 0 & \mathrm{T}_Y(g) \\
               \end{array}
             \right)$$
            
Therefore $g = \T_N( \beta_{1}) = \T_N(h)\T_N(f^{0})$,  
which implies that the sequence (3.6) is $\Hom_{\Gamma}(-$, $\T_N(C))$-exact.
Continuing the process, we obtain the assertion.  
\end{proof}

\begin{theorem} 
Let $\Lambda$ and $\Gamma$ be separably equivalent linked by $_{\Lambda}M_\Gamma$ and $_\Gamma N_\Lambda$,
and let $G$ be a $\Lambda$-module. 

$(1)$ If $G\in$ $r\mathcal{G}(\C)$, then $\T_{N}(G) \in$ $r\mathcal{G}(\T_{N}(\C))$.

$(2)$ If $H\in$ $l\mathcal{G}(\C)$, then $\T_{N}(H) \in$ $l\mathcal{G}(\T_{N}(\C))$.
\end{theorem} 

\begin{proof}
By the definition of one-sided Gorenstein category, the assertion follows from Lemma 3.1 and 3.2.
\end{proof}

The following result shows that Gorenstein categories $\mathcal{G}(\C)$ are preserved 
under any separable equivalence.

\begin{theorem}
Let $\Lambda$ and $\Gamma$ be separably equivalent linked by $_{\Lambda}M_\Gamma$ and $_\Gamma N_\Lambda$.
If $G\in$ $\mathcal{G}(\C)$, then $\T_{N}(G) \in$ $\mathcal{G}(\T_{N}(\C))$.
\end{theorem}

\begin{proof}
Since $G\in$ $\mathcal{G}(\C)$, 
there exist exact sequences: $$0\to G\stackrel{f^{0}}\to C^{0}\stackrel{f^{1}}\to C^{1}\to \cdots \qquad(3.7)$$
and $$\cdots \to C_1 \stackrel{f_{1}}\to C_0\stackrel{f_{0}} \to G\to 0\qquad(3.8)$$ with all $C^{i}, C_{j}\in\C$, 
which are both $\Hom_{\Lambda}(\C, -)$-exact and $\Hom_{\Lambda}(-, \C)$-exact.
It follows from Lemma 3.2 that the exact sequence below
$$0\to \T_{N}(G)\stackrel{\T_{N}(f^{0})}\to \T_{N}(C^{0})\stackrel{\T_{N}(f^{1})}\to \T_{N}(C^{1})\to \cdots \qquad(3.9)$$  is $\Hom_{\Gamma}(-, \T_{N}(\C))$-exact.
We claim that the sequence (3.9) is also $\Hom_{\Gamma}(\T_{N}(\C), -)$-exact.
In fact, for each $i\geq 1$, let $K^{i} = \Im f^{i}$, then $ \Im \T_{N}(f^{i})\cong$ $\T_{N}(K^{i})$.
Consider the exact sequence
$$0\to \T_{N}(G)\stackrel{\T_{N}(f^{0})}\rightarrow \T_{N}(C{^{0}})\stackrel{\T_{N}(\pi)} \to \T_{N}(K^{1})\to 0 \qquad(3.10)$$

Given $C\in\C$ and $g\in\Hom_{\Gamma}(\T_{N}(C), \T_{N}(K^{1}))$,
one gets \begin{center}
$\T_{M}(g)\in$ $\Hom_{\Lambda}(\T_{M}\T_{N}(C), \T_{M}\T_{N}(K^{1}))$ $\cong$
$\Hom_{\Lambda}(C$ $\oplus H(C), K^{1}\oplus H(K^{1})$.
\end{center}
Write $\T_{M}(g) = \left(\begin{array}{cc}
                \gamma_{1} & \gamma_{2} \\
                 \gamma_{3} & \gamma_{4} \\
               \end{array}
             \right)$, then $\gamma_{1}\in$ $\Hom_{\Lambda}(C, K^{1})$.
Because the sequence (3.7) is $\Hom_{\Lambda}(C$, $-)$-exact,
$\Hom_{\Lambda}(C, \pi): \Hom_{\Lambda}(C, C^{0})\rightarrow \Hom_{\Lambda}(C, K^{1})\to 0$ is surjective.
Thus there exists $h\in$ $\Hom_{\Lambda}(C, C^{0})$ such that $ \gamma_{1} = \pi h$.

On the other hand, we have
 $$\T_{N}\T_{M}(g)=
\left(\begin{array}{cc}
                g & 0 \\
                 0 & \T_{Y}(g) \\
               \end{array}
             \right)$$
            
Therefore $g = \T_{N}( \gamma_{1}) = \T_{N}(\pi)\T_{N}(h)$,  
which implies that the short exact sequence (3.10) is $\Hom_{\Gamma}(\T_{N}(C)$, $-)$-exact. Repeating this process, we obtain the claim.

Similarly, we have the exact sequence $$\cdots \to \T_{N}(C_1) \stackrel{\T_{N}(f_{1})}\to \T_{N}(C_0)\stackrel{\T_{N}(f_{0})} \to \T_{N}(G)\to 0$$
which is both $\Hom_{\Gamma}(\T_{N}(\C), -)$-exact and $\Hom_{\Gamma}(-, \T_{N}(\C))$-exact.
Therefore, $\T_{N}(G) \in$ $\mathcal{G}(\T_{N}(\C))$.
\end{proof}

\begin{lemma}
Let $\Lambda$ and $\Gamma$ be separably equivalent defined by $_{\Lambda}M_\Gamma$ and $_\Gamma N_\Lambda$.
Given a subcategory $\D$ of Mod $\Gamma$ and  $L\in$ Mod $\Gamma$,
assume $\T_{N}(\C)\subseteq\D$, $\T_{M}(\D)\subseteq\C$.

(1) If $L\in$ $^{\bot}(\T_{N}(\C))$, %$($resp. $(\T_{N}(\C))^{\bot})$, 
then $L\in$ $^{\bot}\D$.
%$($resp. $\D^{\bot})$.

(2) If $\C$ is self-orthogonal and  $L \in$ $\mathrm{cores} \widetilde{\T_{N}(\C)}$, 
%$($resp. $\mathrm{res} \widetilde{\T_{N}(\C)})$, 
then $L\in$ $\mathrm{cores} \widetilde{\D}$. 
%$($resp. $\mathrm{res} \widetilde{\D})$.
\end{lemma}

\begin{proof}
(1) For any $D\in\D$, by assumption, $\T_{M}(D)\in\C$, and so $\T_{N}\T_{M}(D)$ is in $\T_{N}(\C)$.
Because $L\in$ $^{\bot}(\T_{N}(\C))$,
$\Ext^{i}_{\Gamma}(L, \T_{N}\T_{M}(D)) = 0$,
which implies $\Ext^{i}_{\Gamma}(L, D) = 0$ since $D$ is a direct summand of $\T_{N}\T_{M}(D)$. Thus $L\in$ $^{\bot}\D$.

(2) Let $L\in$ $\mathrm{cores} \widetilde{\T_{N}(\C)}$.
Then there is a $\Hom_{\Gamma}(-, \T_{N}(\C))$-exact exact sequence 
$$0\to L\to D^{0}\to D^{1}\to\cdots \qquad (3.7)$$
with each $D^{i}\in$ $\T_{N}(\C)\subseteq\D$.
Since $\C$ is self-orthogonal, so is $\T_{N}(\C)$ by Lemma 3.1.
It follows from \cite[Lemma 4.1.1]{C} that every $L^{i} = \Im(D^{i-1}\to D^{i})$ is in $^{\bot}(\T_{N}(\C))$ for $i\geq 1$.
By (1) one gets $L^{i}\in$ $^{\bot}\D$,
which shows that the sequence (3.7) is $\Hom_{\Gamma}(-, \D)$-exact.
Hence $L\in$ $\mathrm{cores} \widetilde{\D}$.
\end{proof}

The following result means that Gorenstein projective and Gorenstein injective modules are invariant under separable equivalences, which improves Proposition 4.5 in \cite{LX}.

\begin{corollary}
Let $\Lambda$ and $\Gamma$ be separably equivalent defined 
by $_{\Lambda}M_\Gamma$ and $_\Gamma N_\Lambda$, then 

$(1)$ $G\in$ $\mathcal{GP}(\Lambda)$ if and only if $\T_{N}(G) \in$ $\mathcal{GP}(\Gamma)$.

$(2)$ $G\in$ $\mathcal{GI}(\Lambda)$ if and only if $\T_{N}(G) \in$ $\mathcal{GI}(\Gamma)$.
\end{corollary}

\begin{proof} 
We only prove (1), and (2) can be proved dually.

Assume $G\in$ $\mathcal{GP}(\Lambda)=$ $r\G(\mathcal{P}(\Lambda))$.
Thus $\T_{N}(G) \in$ $r\mathcal{G}(\T_{N}(\mathcal{P}(\Lambda)))$ by Theorem 3.3.
Since $\T_{N}(\mathcal{P}(\Lambda)) \subseteq$ $\mathcal{P}(\Gamma)$ and 
$\T_{M}(\mathcal{P}(\Gamma)) \subseteq$ $\mathcal{P}(\Lambda)$ by Lemma 2.3(2),
then $r\mathcal{G}(\T_{N}(\mathcal{P}(\Lambda)))$ $\subseteq$ $r\mathcal{G}(\mathcal{P}(\Gamma))$ 
from Lemma 3.5, which yields $\T_{N}(G) \in$ $\mathcal{GP}(\Gamma)$.

Conversely, assume $\T_{N}(G) \in$ $\mathcal{GP}(\Gamma)$.
By a similar argument of the \lq\lq only if\rq\rq part, we have 
$\T_{M}\T_{N}(G) \in$ $\mathcal{GP}(\Lambda)$,
and so $G \in$ $\mathcal{GP}(\Lambda)$ from \cite[Theorem 2.5]{Ho}.
\end{proof} 

\begin{corollary}
Let $\Lambda$ and $\Gamma$ be separably equivalent defined 
by $_{\Lambda}M_\Gamma$ and $_\Gamma N_\Lambda$, then 

$(1)$ $A\in$ $^{\bot}\mathcal{P}(\Lambda)$ if and only if $\T_{N}(A) \in$ $^{\bot}\mathcal{P}(\Gamma)$.

$(2)$ $A\in$ $\mathcal{T}^{\infty}(\Lambda)$ if and only if $\T_{N}(A) \in$ $\mathcal{T}^{\infty}(\Gamma)$.
\end{corollary}

\begin{proof} 
(1) Let $\C = \mathcal{P}(\Lambda)$ in Lemma 3.1 (1), 
$A\in$ $^{\bot}\mathcal{P}(\Lambda)$ implies  $\T_{N}(A) \in$ $^{\bot}\T_{N}(\mathcal{P}(\Lambda))$.
Since $\T_{M}(\mathcal{P}(\Gamma))\subseteq$ $\mathcal{P}(\Lambda)$,
then $\T_{N}(A) \in$ $^{\bot}\mathcal{P}(\Gamma)$ from Lemma 3.5 (1).
Conversely, if $\T_{N}(A) \in$ $^{\bot}\mathcal{P}(\Gamma)$, then 
$\T_{M}\T_{N}(A) \in$ $^{\bot}\mathcal{P}(\Lambda)$, and so is $A$ since $A$ is a direct summand of $\T_{M}\T_{N}(A)$.

(2) Let $A\in$ $\mathcal{T}^{\infty}(\Lambda)$, that is $A\in$ $\mathrm{cores} \widetilde{\mathcal{P}(\Lambda)}$.
Thus $\T_{N}(A) \in$ $\mathrm{cores} \widetilde{\T_{N}(\mathcal{P}(\Lambda))}$.
Lemma 3.5 (2) yields that $\T_{N}(A) \in$ $\mathrm{cores} \widetilde{\mathcal{P}(\Gamma)}$ = $\mathcal{T}^{\infty}(\Gamma)$. The converse follows from the fact that $\mathrm{cores} \widetilde{\C}$ are closed under direct summands.

\end{proof}

\begin{lemma}
Let $\Lambda$ and $\Gamma$ be separably equivalent defined by $_{\Lambda}M_\Gamma$ and $_\Gamma N_\Lambda$. For any $L\in$ Mod $\Gamma$, one has

(1) If $L\in$ $^{\bot}(\T_{N}(\Add C))$, then $L\in$ $^{\bot}(\Add \T_{N}(C))$.

(2) If $L\in$ $\mathrm{cores} \widetilde{\T_{N}(\Add C)}$, then $L\in$ $\mathrm{cores}( \widetilde{\Add \T_{N}(C)})$.

In particular, the results above also hold for the case of $\add$.
\end{lemma}

\begin{proof}
(1) Assume $L\in$ $^{\bot}(\T_{N}(\Add C))$.
Because $\Ext^{i}_{\Gamma}(L, \oplus\T_{N}(C)) \cong$ $\Ext^{i}_{\Gamma}(L, \T_{N}(\oplus C)) =0$,
one gets $L\in$ $^{\bot}(\Add \T_{N}(C))$.

(2) Let $L\in$ $\mathrm{cores} \widetilde{\T_{N}(\Add C)}$.
Then there is a $\Hom_{\Gamma}(-, \T_{N}(\Add C))$-exact exact sequence 
$$0\to L\to D^{0}\to D^{1}\to\cdots \qquad (3.7)$$
with each $D^{i}\in$ $\T_{N}(\Add C)$. Clearly, $\T_{N}(\Add C)\subseteq \Add \T_{N}(C)$.
Since $\Add C$ is self-orthogonal from \cite[Proposition 2.6 and 2.7]{Wh}, so is $\T_{N}(\Add C)$ by Lemma 3.1.
It follows from \cite[Lemma 4.1.1]{C} that every $L^{i} = \Im(D^{i-1}\to D^{i})$ is in $^{\bot}(\T_{N}(\Add C))$ for $i\geq 1$.
By (1) one gets $L^{i}\in$ $^{\bot}\Add \T_{N}(C)$,
which shows that the sequence (3.7) is $\Hom_{\Gamma}(-, \Add \T_{N}(C))$-exact,
and hence we obtain the assertion.
\end{proof}

\begin{proposition} 
Let $\Lambda$ and $\Gamma$ be separably equivalent linked by $_{\Lambda}M_\Gamma$ and $_\Gamma N_\Lambda$.
If $C$ is a Wakamatsu tilting $\Lambda$-module, then $N\otimes _{\Lambda} C$ is a Wakamatsu tilting $\Gamma$-module.
\end{proposition} 

\begin{proof}
Assume that $C$ is a Wakamatsu tilting $\Lambda$-module.
First of all, since $C$ is self-orthogonal, so is $\T_{N}(C)$ by Lemma 3.1.
Because $C\in$ $\mathrm{res} \widetilde{(\add \Lambda)}$, there exists an exact sequence 
$\cdots\to P_{1}\to P_0\to C\to 0$ with $P_{i}\in \add \Lambda$. 
Applying $\T_N$ to this sequence gives rise to an exact sequence
$$\cdots\to \T_N(P_{1})\to \T_N(P_0)\to \T_N(C)\to 0$$
Obviously, $\T_{N}(P_{i})\in \add \Gamma$, which means 
$\T_{N}(C)\in$ $\mathrm{res} \widetilde{(\add \Gamma)}$.

Since $\Lambda\in \mathrm{cores} \widetilde{(\add C)}$, Lemma 3.2 yields
$_{\Gamma}N\cong\T_{N}(\Lambda)\in$ $\mathrm{cores} \widetilde{(\T_{N}(\add C))}$.
From the {\lq add\rq} case of Lemma 3.8(2), one has $N\in$ $\mathrm{cores} \widetilde{(\add \T_{N}(C))}$.
Since $_{\Gamma}N$ is a projective generator due to Lemma 2.3(1), then 
$\Gamma\in\add{_\Gamma}N$. By a dual result of \cite[Lemma 4.2]{HS}, 
one has $\Gamma\in$ $\mathrm{cores} \widetilde{(\add \T_{N}(C))}$,
which completes the proof.
\end{proof}

\begin{corollary}
Let $\Lambda$ and $\Gamma$ be separably equivalent defined by $_{\Lambda}M_\Gamma$ and $_\Gamma N_\Lambda$.

$(1)$ If $G\in$ $\mathcal{GP}_{C}(\Lambda)$, then $\T_{N}(G) \in$ $\mathcal{GP}_{\T_{N}(C)}(\Gamma)$.

$(2)$ If $H\in$ $\mathscr{B}_{C}(\Lambda)$, then $\T_{N}(H) \in$ $\mathscr{B}_{\T_{N}(C)}(\Gamma)$.
\end{corollary}

\begin{proof} 
(1) Let $G\in$ $\mathcal{GP}_{C}(\Lambda)$.
From Remark 2.6(2) we know that $\mathcal{GP}_{C}(\Lambda) = $ $r\mathcal{G}(\Add C)$.
Thus $\T_{N}(G)\in$ $r\mathcal{G}(\T_{N}(\Add C))$ $\subseteq$
$r\mathcal{G}(\Add \T_{N}(C))$ by Theorem 3.3(1) and Lemma 3.8. 
Note that $\T_{N}(C)$ is a Wakamatsu tilting $\Gamma$-module from Proposition 3.9,
hence $\T_{N}(G)\in r\mathcal{G}(\Add \T_{N}(C))$ $=\mathcal{GP}_{\T_{N}(C)}(\Gamma)$. 

(2) Note that $\mathscr{B}_{C}(\Lambda) = $ $l\mathcal{G}(\Add C)$.
The proof is similar to (1) by using the dual result of Lemma 3.8.
\end{proof}

\begin{lemma}
Let $\Lambda$ and $\Gamma$ be separably equivalent defined by $_{\Lambda}M_\Gamma$ and $_\Gamma N_\Lambda$. 
For any $L\in$ Mod $\Gamma$, one has

(1) If $L\in$ $^{\bot}(\T_{M}(\Prod C^{+}))$, then $L\in$ $^{\bot}(\Prod \T_{M}(C^{+}))$.

(2) If $L\in$ $\mathrm{cores} (\widetilde{\T_{M}(\Prod C^{+})})$, then $L\in$ $\mathrm{cores}( \widetilde{\Prod \T_{M}(C^{+})})$.
\end{lemma}

\begin{proof}
Since $M_\Gamma$ is finitely generated projective, then 
$\T_{M}$ commutes with direct products and
$\T_{M}(\Prod C^{+})\subseteq$ $\Prod \T_{M}(C^{+})$.
By a similar argument of Lemma 3.8, one gets the assertion.
\end{proof}

\begin{proposition}
Let $_{\Lambda}C$ be a Wakamatsu tilting module with $\Gamma = \End(_{\Lambda}C)$,
and suppose that $\Lambda$ and $\Gamma$ are separably equivalent defined by 
$_{\Lambda}M_\Gamma$ and $_\Gamma N_\Lambda$ with  an adjoint pair $(-\otimes_{\Lambda}M, -\otimes_{\Gamma}N)$.

$(1)$ If $H\in$ $\mathscr{A}_{C}(\Gamma)$, then $\T_{M}(H) \in$ $\mathscr{A}_{C\otimes_{\Gamma}N}(\Lambda)$.

$(2)$ If $G\in$ $\mathcal{GI}_{C}(\Gamma)$, then $\T_{M}(G) \in$ $\mathcal{GI}_{C\otimes_{\Gamma}N}(\Lambda)$.
\end{proposition}

\begin{proof} 
(1) Let $H\in$ $\mathscr{A}_{C}(\Gamma)$.
From Remark 2.6(2) we know that $\mathscr{A}_{C}(\Gamma) = $ $r\mathcal{G}(\Prod C^{+})$.
Thus $\T_{M}(H)\in$ $r\mathcal{G}(\T_{M}(\Prod C^{+}))$ by Theorem 3.3(1).
From Lemma 3.11, one has $r\mathcal{G}(\T_{M}($ $\Prod C^{+}))$
$\subseteq$ $r\mathcal{G}(\Prod \T_{M}(C^{+}))$.

Because $(-\otimes_{\Lambda}M, -\otimes_{\Gamma}N)$ is an adjoint pair,
$\T_{M}(C^{+})\cong(\Hom_{\Gamma^{op}}(M, C))^{+}$ $\cong$ $(C\otimes_{\Gamma}N)^{+}$.
Note that $C_{\Gamma}$ is Wakamatsu tilting from \cite[Corollary 3.2]{Wa}, 
it follows from Proposition 3.9 that $C\otimes_{\Gamma}N$ is a Wakamatsu tilting right $\Lambda$-module.
Therefore, $\T_{M}(H)\in$ $r\mathcal{G}(\Prod (C\otimes_{\Gamma}N)^{+})$ $=\mathscr{A}_{C\otimes_{\Gamma}N}(\Lambda)$.

(2) Note that $\mathcal{GI}_{C}(\Gamma) = $ $l\mathcal{G}(\Prod C^{+})$.
The proof is similar to (1) by using the dual result of Lemma 3.11.
\end{proof}

\vspace{0.2cm}

\noindent {\bf Acknowledgements}

This research was partially supported by NSFC (Grant No. 12061026) and Foundation
for University Key Teacher by Henan Province (2019GGJS204).

\vspace{0.5cm}

%Question: Rep dim, Goren tilting are invariant under separa equi?

\end{document}